\documentclass[10pt,reqno,a4]{amsart}
\usepackage[utf8]{inputenc}
\usepackage{graphicx}
\usepackage{tabularx} 
\usepackage{amsmath,amsfonts,amssymb,mathrsfs,amsthm}
\usepackage{mathabx}
\usepackage{subcaption}
\usepackage{stmaryrd}
\usepackage{enumitem}
\usepackage{here}
\usepackage{xcolor}
\usepackage[colorlinks=true]{hyperref}

\newtheorem{theo}{Theorem}[section]
\newtheorem{prop}[theo]{Proposition}
\newtheorem{lemme}[theo]{Lemma}

\newtheorem{remarque}[theo]{Remark}

\newtheorem{cor}[theo]{Corollary}

\newcommand{\tran}{\mathsf{T}} 
\newcommand{\bs}{\boldsymbol}

\newcommand{\maxime}[1]{\textcolor{magenta}{[Maxime: #1]}\color{black}} 
\newcommand{\hafedh}[1]{\textcolor{cyan}{[Hafedh: #1]}\color{black}}
\newcommand{\jamal}[1]{\textcolor{red}{[Jamal: #1]}\color{black}}

\newcommand{\x}{\boldsymbol{x}}

\title[Equilibrium in a large Lotka-Volterra system]{Equilibrium in a large Lotka-Volterra system with pairwise correlated interactions}

\author{Maxime Clenet, Hafedh El Ferchichi, Jamal Najim }
\date{\today}
\thanks{Supported by CNRS Project 80 Prime | KARATE.}


\usepackage{nomencl}

\renewcommand{\nompreamble}{The next list describes several symbols that will be later used within the body of the document}
\makenomenclature

\begin{document}
\maketitle
\mbox{}
\nomenclature{\(A_n\)}{Interaction matrix of size $n$, following the elliptic model, of coefficient $A_{kl}$}
\nomenclature{\(B^\top\)}{Transpose of matrix $B$}
\nomenclature{\(\mathbf{1}_n\)}{vector of size n, with all coefficients equal to 1}
\nomenclature{\(\alpha_n\)}{normalizing sequence, typically goes to $+\infty$}
\nomenclature{\(\alpha_n^* \)}{critical sequence, equals to $\sqrt{2\log n}$}
\nomenclature{\(\mathbf{x}\)}{Density vector, each coefficients $x_k$ representing the abondance of some species}
\nomenclature{\(\mathbb{P}(A)\)}{Probability measure of some event $A$}
\nomenclature{\(\mathbb{E}(X)\)}{Expectation of $X$}
\nomenclature{\(I_n\)}{Identity matrix of size $n$}
\nomenclature{\(\mathbf{e}_k\)}{k-th vector of the canonical base}
\nomenclature{\(\xrightarrow[n\to\infty]{\quad \mathcal D\quad}\)}{Convergence in distribution}
\nomenclature{\(\rho\)}{Co-variance parameter in the elliptic model}
\nomenclature{\(\mu\)}{Mean parameter in the elliptic model}
\nomenclature{\(G(x)\)}{Gumbel distribution, evaluated at x: $G(x) = e^{-e^{-x}}$}
\nomenclature{\(\xrightarrow[n \to \infty]{\mathcal{P}}\)}{Convergence in probability}
\nomenclature{\(cov(X,Y)\)}{Covariance of X and Y}
\nomenclature{\(\mathbb{R},\mathbb{C}\)}{real/complex number}
\nomenclature{\(\mathcal{O}, o, \sim\)}{Standard Landau notations}
\nomenclature{\(\mathcal{N}(\mu, \rho)\)}{Normal distribution, of mean $\mu$, and variance $\rho$}
\nomenclature{\(B^\top\)}{Transpose of matrix $B$}
\nomenclature{\(B^\top\)}{Transpose of matrix $B$}


\begin{abstract}
Consider a Lotka-Volterra (LV) system of coupled differential equations:
$$
\dot{x}_k= x_k(r_k - x_k + (B\bs{x})_k)\,,\quad  \bs{x}=(x_k)\,,\quad  1\le k\le n\, , 
$$
where $\bs{r}=(r_k)$ is a $n\times 1$ vector and $B$ a $n\times n$ matrix. Assume that the interaction matrix $B$ is random and follows the elliptic model:
$$
B=\frac 1{\alpha\sqrt{n}} A + \frac{\mu}{n} \mathbf{1}_{n}\mathbf{1}_{n}^\tran\ ,
$$
where $A=(A_{ij})$ is a $n\times n$ matrix with ${\mathcal N}(0,1)$ entries satisfying the following dependence structure $(i)$ the entries $A_{ij}$ on and above the diagonal are i.i.d., $(ii)$ for $i< j$ each vector $(A_{ij}, A_{ji})$ is standard gaussian with covariance $\rho$, and independent from the other entries; vector $\mathbf{1}_{n}$ stands for the $n\times 1$ vector of ones. Parameters $\alpha,\mu$ are deterministic and may depend on $n$.

Leveraging on Random Matrix Theory, we analyse this LV system as $n\to \infty$ and study the existence of a positive equilibrium. This question boils down to study the existence of a (componentwise) positive solution to the linear equation:
$$
\x_n = \bs{r}_n + B_n \x_n\, ,
$$
depending on $B$'s parameters $(\alpha, \mu,\rho)$, a problem of independent interest in linear algebra.

In the case where no positive equilibrium exists, we provide sufficient conditions for the existence of a unique stable equilibrium (with vanishing components), and following Bunin \cite{bunin2017ecological}, present a heuristics estimating the number of positive components of the equilibrium and their distribution.     

The existence of positive equilibria for large Lotka-Volterra systems has been raised in Dougoud et al. \cite{dougoud2018feasibility}, and addressed in various contexts by Najim et al. \cite{akjouj2021feasibility,bizeul2021positive}.

Such LV systems are widely used in mathematical biology to model populations with interactions, and the existence of a positive equilibrium known as a {\em feasible equilibrium} corresponds to the survival of all the species $x_k$ within the system.

\end{abstract}

\section{Introduction}

\subsection*{Lotka-Volterra system of coupled differential equations.} Lotka-Volterra (LV) systems are widely used in mathematical biology, ecology, chemistry to model populations with interactions or chemical reactions \cite{gopalsamy1984global,hofbauer1998evolutionary,kiss2008qualitative,hering1990oscillations}.
In the context of theoretical ecology (that we shall adopt hereafter without loss of generality), consider a given foodweb and denote by $\x_n^t=(x_k(t))_{1\le k\le n}$ the vector of abundances\footnote{A species abundance is a quantity proportional to the number of individuals for this species.} of the various species at time $t\ge 0$. 
In a LV system, the abundances are connected via the following coupled equations:
$$
\frac{dx_k(t)}{dt} = x_k(t)\, \left( r_k - x_k(t) +  \sum_{\ell=1}^n B_{k\ell} x_{\ell}(t)\right)\qquad \textrm{for}\quad k\in [n]:=\{1,\cdots, n\}\, ,
$$
where $B_n=(B_{k\ell})$ stands for the interaction matrix, and $r_k$ stands for the intrinsic growth of species $k$.
Notice that standard results yield that if the initial condition $\bs{x}_n^0=\bs{x}_n|_{t=0}$ is componentwise positive, then $\bs{x}_n^t$ remains positive for every $t>0$. 

At the equilibrium $\frac{d\x_n}{dt}=0$, the abundance vector $\x_n=(x_k)_{k\in [n]}$ is solution of the system:
\begin{equation}\label{eq:equilibrium}
x_k\, \left( r_k - x_k +  \sum_{\ell\in[n]} B_{k\ell} x_{\ell}\right)=0\qquad \textrm{for}\quad x_k\ge 0\quad \text{and}\quad k\in [n]\ .
\end{equation}

An important question, which motivated recent developments \cite{dougoud2018feasibility,bizeul2021positive}, is the existence of a {\it feasible} solution $\x_n$ to \eqref{eq:equilibrium}, that is a solution where all the $x_k$'s are positive, corresponding to a scenario where no species disappears. Notice that in this latter case, the system \eqref{eq:equilibrium} takes the much simpler form:
$$
\x_n = \boldsymbol{r}_n + B_n \x_n\ ,\qquad \boldsymbol{r}_n=(r_k)_{k\in [n]}\, .
$$
In this article, we will investigate the existence of an equilibrium, potentially feasible, for a large foodweb ($n\to \infty$) whenever the interaction matrix $B_n$ is random. In various models of interest for $B_n$, Random Matrix Theory (RMT) provides an accurate description of the asymptotic properties of a large random matrix (its spectrum, spectral norm, etc.). We will leverage on RMT to infer the existence of an equilibrium in the case where $B_n$ follows a random elliptic model, to be described hereafter.  

To simplify the analysis, we will consider the case where $\boldsymbol{r}_n=\bs{1}_n$.
\subsection*{Random elliptic model for the interaction matrix} In the spirit of May\footnote{Beware that May did not consider LV systems but rather used a random matrix model for the Jacobian at equilibrium of a generic system of coupled differential equations.}, we model the interaction matrix $B_n$ as a non-centered random matrix with pairwise correlated entries:
\begin{equation}\label{eq:elliptic-model}
    B_n = \frac{A_n}{\alpha_n \sqrt{n}}+\frac{\mu}n \boldsymbol{1}_n \boldsymbol{1}^\tran_n \, ,
\end{equation}
where $A_n=(A_{ij})_{i,j\in [n]}$ is a random matrix satisfying the two conditions $(i)$ $(A_{ij}, i\le j)$ are standard Gaussian ${\mathcal N}(0,1)$ independent and identically distributed (i.i.d.) random variables $(ii)$ for $i<j$ the vector $(A_{ij},A_{ji})$ is a standard bivariate Gaussian vector, independent from the remaining random variables, with covariance $\mathrm{cov}(A_{ij}, A_{ji})=\rho$ with $|\rho|\le 1$. The sequence of positive numbers $(\alpha_n)$ is either fixed or goes to infinity. Parameter $\mu$ is a fixed real number. As a consequence, the Gaussian entries of the interaction matrix $B_n$ admit the following moments:
$$
\mathbb{E}(B_{ij}) = \frac {\mu}n\,,\quad \mathrm{var}(B_{ij}) = \frac 1{\alpha^2 n}\,,\quad \mathrm{cov}(B_{ij},B_{ji})=\frac \rho{\alpha^2 n}\  (i\neq j)\,.
$$
Such a matrix model is often called a random elliptic model for $|\rho|<1$ since the spectrum of matrix $A_n/\sqrt{n}$ is asymptotically an ellipse, see Fig.\ref{fig:elliptic-model}, in the sense that the empirical distribution of the eigenvalues of $A_n/\sqrt{n}$ converges towards the uniform distribution on the ellipsoid $${\mathcal E}_\rho=\left\{
z\in \mathbb{C},\ \frac{\mathrm{Re}^2(z)}{(1+\rho)^2} + \frac{\mathrm{Im}^2(z)}{(1-\rho)^2} \le 1
\right\}\, .$$
Originally introduced by Girko \cite{girko1986elliptic}, this model has since been widely studied \cite{girko1995elliptic,naumov2012elliptic,nguyen2015elliptic,orourke2014low}. 

\begin{figure}[htb]
\centering
\begin{subfigure}[b]{.32\textwidth}
    \includegraphics[width=\textwidth]{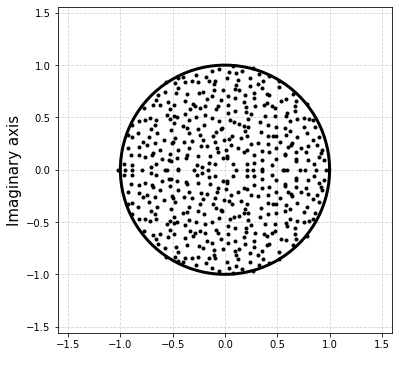}
    \caption{$\rho=0$}
    \label{subfig:SC}
  \end{subfigure}%
  \hfill   
  \begin{subfigure}[b]{0.32\textwidth}
    \includegraphics[width=\textwidth]{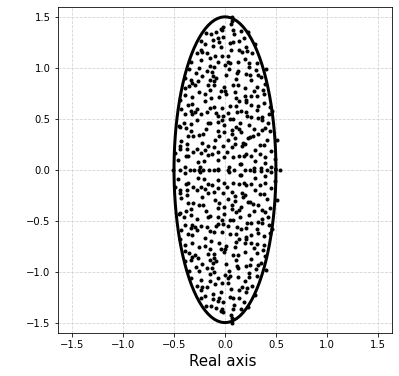}
    \caption{$\rho=-0.5$}
    \label{subfig:SC-outlier}
  \end{subfigure}%
\hfill   
  \begin{subfigure}[b]{0.32\textwidth}
    \includegraphics[width=\textwidth]{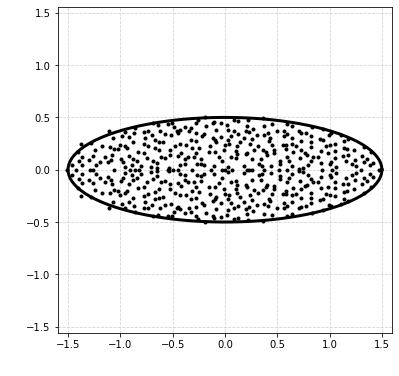}
    \caption{$\rho=0.5$}
    \label{subfig:SC-outlier}
  \end{subfigure}%
\caption{Spectrum of non-Hermitian matrix $B_n$ ($n=500$) in the centered case ($\mu$ = 0) with distinct parameter $\rho \in \{-0.5 ,\ 0 ,\ 0.5\}$. The solid line represents the ellipse $\{ z=x+\bs{i}y\in \mathbb{C},\, \frac{x^2}{(1+\rho)^2} + \frac{y^2}{(1-\rho)^2} =1\} $ which is the boundary of the support of the limiting spectral distribution for an elliptic model.}
\label{fig:elliptic-model}
\end{figure}

The spectral norms of $A_n$ and $\boldsymbol{1}_n \boldsymbol{1}_n^\tran$ satisfy
$$\left\| \frac {A_n}{\sqrt{n}} \right\| ={\mathcal O}\left(1\right)\quad \textrm{and}\quad  \left\| \frac{1}n \boldsymbol{1}_n \boldsymbol{1}_n^\tran\right\|=1$$
hence both the random and deterministic parts of the interaction matrix $B_n$ may have an impact as $n\to\infty$.

\subsection*{Presentation of the main results} In this article, we address the following issues. 

\subsubsection*{Feasibility.} We first describe the conditions over parameters $(\rho, \alpha_n, \mu)$ for which system \eqref{eq:equilibrium} admits a unique feasible equilibrium. We prove that feasibility is reached whenever $\alpha_n\gg \sqrt{2\log(n)}$ and $\mu<1$, and that there is no feasibility otherwise, see Theorem \ref{th:feasibility}. Notice that the correlation parameter $\rho$ has no influence since the phase transition threshold is the same as in the i.i.d. case \cite{bizeul2021positive}: the induced correlations between components $x_k$'s of solution $\bs{x}_n$ are too weak. Pushing this remark further, we prove that the same phase transition holds if we consider a covariance profile $(\rho_{ij},\, i<j)$ where $\rho_{ij}=\mathrm{cov}(A_{ij},A_{ji})$ instead of a fixed covariance parameter $\rho$.

 In \cite{bizeul2021positive}, Bizeul and Najim established the conditions for feasibility in the centered ($\mu=0$) model with i.i.d interactions $(A_{ij})$. In \cite{akjouj2021feasibility}, Akjouj and Najim studied a sparse model of interactions where only $d_n\ge \log(n)$ interactions are non-null in each row and column of $A_n$. The study of the feasibility for an elliptic model completes this picture.

\subsubsection*{Stability without feasibility.} 
If $\alpha$ is fixed, Dougoud et al. \cite{dougoud2018feasibility} showed that no feasible solution can arise.
Under this assumption, we establish in Proposition \ref{prop:unique-equilibrium} sufficient conditions for the existence of a unique stable equilibrium to system \eqref{eq:equilibrium}. In this case, some species will vanish (some of the components $x_k$'s of solution $\bs{x}_n$ are equal to zero). In order to proceed we combine results by Takeuchi \cite{takeuchi1996global} on stability of LV systems with Random Matrix Theory (RMT) results.

\subsubsection*{Estimating the number of surviving species.}  We finally conclude with an important question: given a set of parameters $(\rho,\alpha,\mu)$ which yields to a unique stable equilibrium, is it possible to estimate the proportion of surviving species? From a mathematical point of view, this is an open question. At a physical level of rigor, Bunin \cite{bunin2017ecological} (relying on the cavity method) and Galla \cite{galla_dynamically_2018} (relying on generating functionals techniques) provide a closed-form system of equations to compute the proportion of surviving species. We state the open problem, recall Bunin's and Galla's equations and provide simulations. 

In \cite{clenet2022preprint}, equations and simulations are provided in the simpler case where $\rho=0$, together with heuristics supporting these equations.

\subsection*{Organisation of the article} Feasibility and stability results together with the open question on the estimation of the number of surviving species are presented in Section \ref{sec:main-results}. Section \ref{sec:proof-feasibility} is devoted to the proof of the feasibility result, Theorem \ref{th:feasibility}. Proof of the stability result, Proposition \ref{prop:unique-equilibrium}, is provided in Section \ref{sec:proof-stability}. Simulations were performed in Python. All the figures and the code are available
on Github \cite{code}.   

\subsection*{Notations} If $A$ is a matrix $A^\tran$ stands for its transpose. We denote by $\log(x)$ the natural logarithm. If $\bs{x}=(x_i)_{i\in [n]}$ is a vector, we denote by $\bs{x}>0$ (resp. $\bs{x}\ge 0$) the componentwise positivity (resp. non-negativity), that is the fact that $x_i>0$ (resp. $x_i\ge 0$) for every $i\in [n]$.

\section{Main results: Feasibility, stability and surviving species} \label{sec:main-results}

\subsection{Feasibility} To simplify the analysis, we consider the case where $r_k=1$ $(k \in [n])$. Hence, the LV system takes the following form in the sequel:
\begin{equation}\label{eq:LV}
\frac{dx_k(t)}{dt} = x_k(t)\, \left( 1 - x_k(t) +  \sum_{\ell\in[n]} B_{k\ell} x_{\ell}(t)\right)\qquad \textrm{for}\quad k\in [n]\, .
\end{equation}
In the next theorem, we describe the conditions to reach a feasible equilibrium. We either assume that matrix $B$ is given by the elliptic model or has a more general covariance profile.

\begin{theo}[Feasibility for the elliptic model] \label{th:feasibility}
Assume that matrix $B_n$ is given by the elliptic model \eqref{eq:elliptic-model}, or that $B_n$ has a covariance profile, i.e.
\begin{equation}\label{eq:cov-model}
B_n =\frac{\tilde A_n}{\alpha_n \sqrt{n}} +\frac{\mu}n \boldsymbol{1}_n \boldsymbol{1}_n^\tran\ ,
\end{equation}
where $\tilde{A}_n$ is a $n\times n$ matrix with entries $(\tilde A_{ij}, i\le j)$ i.i.d. ${\mathcal N}(0,1)$ and where $(\tilde A_{ij}, \tilde A_{ji})$ is a standard bivariate gaussian vector for $i<j$, independent from the remaining random variables, with covariance $\mathrm{cov}(\tilde A_{ij}, \tilde A_{ji})=\rho_{ij}^{(n)}$, 
where $(\rho_{ij}^{(n)}; \, i<j; n\ge 1)$ is a collection of deterministic real numbers in $[-1,1]$.

Let $\alpha_n \xrightarrow[n\to\infty]{} \infty$ and denote by $\alpha_n^*=\sqrt{2\log n}$. If $\mu\neq 1$ then the following equation 
$$
\boldsymbol{x}_n = \mathbf{1}_n + B_n \boldsymbol{x}_n
$$
almost surely admits a unique solution $\boldsymbol{x}_n=(x_k)_{k\in [n]}$.
\begin{enumerate}
\item (feasibility) If $\mu<1$ and there exists $\varepsilon>0$ such that, for $n$ large enough, $\alpha_n\ge (1+\varepsilon)\alpha_n^*$ then
$$
\mathbb{P}\left\{ \min_{k\in [n]} x_k>0\right\} \xrightarrow[n\to\infty]{} 1\, .
$$
\item If $\mu>1$ or there exists $\varepsilon>0$ such that, for $n$ large enough, $\alpha_n\le (1-\varepsilon) \alpha_n^*$ then
$$
\mathbb{P}\left\{ \min_{k\in [n]} x_k>0\right\} \xrightarrow[n\to\infty]{} 0\, .
$$
\end{enumerate}
\end{theo}
\begin{figure}[h!]
  \centering
  \centerline{\includegraphics[width=8.5cm]{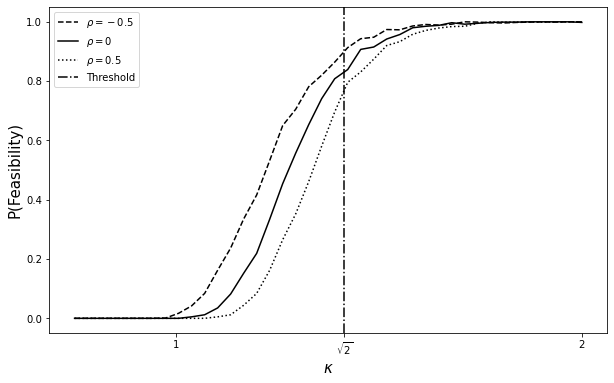}}
\caption{Transition towards feasibility for the elliptic model \eqref{eq:elliptic-model}. For each $\kappa$ on the $x$-axis, we simulate $1000$ matrices $B_n$ of size $n=1000$ and compute the solution $x_n$ of Theorem \ref{th:feasibility} at the scaling $\alpha_n(\kappa) = \kappa \sqrt{\log(n)}$. Each curve represents the proportion of feasible solutions $x_n$ obtained for the $1000$ simulations. Three distinct values $\rho \in \{-0.5 ,\ 0 ,\ 0.5\}$ are used. The dot-dashed vertical line corresponds to $\kappa=\sqrt{2}$ i.e. the critical scaling $\alpha_n^* = \sqrt{2}\sqrt{\log(n)}$. }
\label{fig:transition}
\end{figure}

Proof of Theorem \ref{th:feasibility} is established in Section \ref{sec:proof-feasibility} under the assumption that $B_n$ follows the elliptic model. The adaptations needed to cover the covariance profile case are provided in Appendix \ref{app:proof-cov-profile}.


\subsection{No feasibility but a unique stable equilibrium.}

Aside from the question of feasibility arises the question of \textit{stability}: for a complex system, how likely a perturbation of the solution $\x_n$ at equilibrium  will return to the equilibrium? Gardner and Ashby \cite{gardner1970connectance} considered stability issues of complex systems connected at random. Based on the circular law for large random matrices with i.i.d. entries, May \cite{may1972will} provided a complexity/stability criterion and motivated the systematic use of large random matrix theory in the study of foodwebs, see for instance  Allesina et al. \cite{Allesina:2015ux}. Recently, Stone \cite{stone2018feasibility} and Gibbs et al. \cite{gibbs2018effect} revisited the relation between feasibility and stability. 

For a generic LV system
\begin{equation}\label{eq:generic-LV}
\frac{d\, y_k(t)}{dt} =y_k(r_k +(C\bs{y})_k)\, ,\quad k\in [n]\, ,
\end{equation}
Takeuchi and Adachi provide a criterion for the existence of a unique equilibrium $\bs{y}^*$ and the global stability of LV systems, see Theorem 3.2.1 in \cite{takeuchi1996global}.
\begin{theo}[Takeuchi and Adachi 1980] \label{th:takeuchi} If there exists a positive diagonal matrix $\Delta$ such that $\Delta C + C^\tran \Delta$ is negative definite, there is a unique non-negative equilibrium $\bs{y}^*$ to \eqref{eq:generic-LV}, which is globally stable:
$$
\forall \bs{y}_0>0\,,\quad \begin{cases}
\bs{y}(0)=\bs{y}_0\\
\bs{y}(t)\ \text{satisfies}\ \eqref{eq:generic-LV}
\end{cases}\ ,\quad \bs{y}(t)\xrightarrow[t\to\infty]{} \bs{y}^*\, .
$$ 
\end{theo} 
Combining this result (setting $I-B=-C$) with results from Random Matrix Theory, we can guarantee the existence of a globally stable equilibrium $\bs{x}^*$ of \eqref{eq:LV} for a wide range of parameters $(\rho,\alpha,\mu)$. Denote by
\begin{multline}
\label{eq:def-A}
\mathcal{A}= \bigg\{ (\rho,\alpha,\mu)\in (-1,1) \times (0,\infty) \times \mathbb{R}\,,\\
\ \alpha > \sqrt{2(1+\rho)},\quad \mu< \frac{1}{2}+\frac 12 \sqrt{1-\frac{2(1+\rho)}{\alpha^2}} \bigg \}
\end{multline}
the set of admissible parameters.

\begin{figure}[h!]
  \centering
  \centerline{\includegraphics[width=8.5cm]{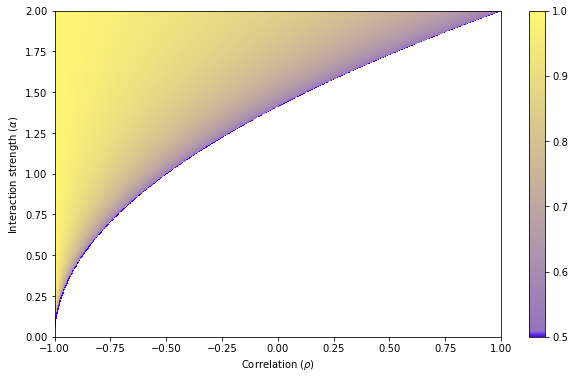}}
   \caption{Representation of the set of admissible parameters $\mathcal{A}$ by a heat map. The set $\mathcal{A}$ given by \eqref{eq:def-A} yields the existence of a unique (random) globally stable equilibrium $\bs{x}^*$. The $x$-axis corresponds to $\rho$, the $y$-axis to $\sigma$ and the intensity of the color $\mu$.}  
\label{fig:admissible}
\end{figure}

\begin{prop}\label{prop:unique-equilibrium} Let $(\rho,\alpha,\mu) \in \mathcal{A}$,
then almost surely, matrix $$(I-B)+(I-B)^\tran$$ is eventually positive definite: with probability one, for a given realization $\omega$, there exists $N(\omega)$ such that for $n\ge N(\omega)$, $(I-B^\omega)+(I-B^\omega)^\tran$ is positive definite. In particular, there exists a unique globally stable non-negative equilibrium $\bs{x}^*$.
\end{prop}

Proof of Proposition \ref{prop:unique-equilibrium} is provided in Section \ref{sec:proof-stability}.

\begin{figure}[htb]
\centering
\begin{subfigure}[b]{.47\textwidth}
    \includegraphics[width=\textwidth]{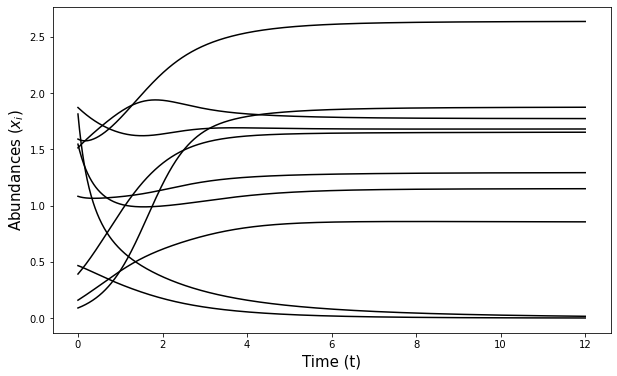}
    \caption{Initial conditions drawn in $(0,2)$,}
    \label{subfig:DL-Diff}
  \end{subfigure}%
  \hfill   
  \begin{subfigure}[b]{0.47\textwidth}
    \includegraphics[width=\textwidth]{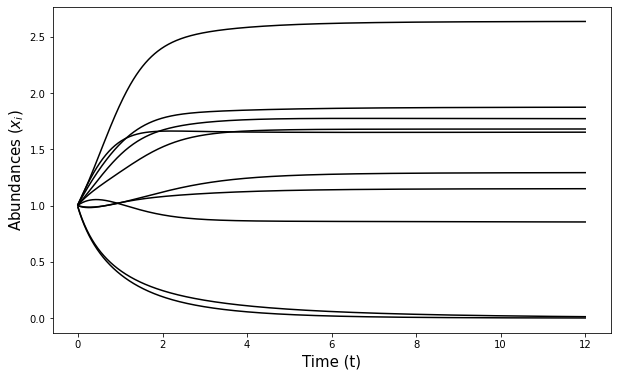}
    \caption{Initial conditions equal to 1.}
    \label{subfig:DL-Same}
  \end{subfigure}%
\caption{Representation of the dynamics of a ten-species system. For a fixed matrix of interactions $B_{10}$ with parameters $(\rho=0,\alpha=2,\mu=0)\in {\mathcal A}$, we consider two distinct initial conditions. Simulations show that the abundances converge in both cases toward the unique globally stable equilibrium $x^*$ predicted by Proposition \ref{prop:unique-equilibrium}. Notice that since $\alpha<\sqrt{2\log(10)}\simeq 2.14$, we witness vanishing species.} 
\label{fig:dynamics-LV}
\end{figure}

\subsection{Estimating the number of surviving species: Towards Bunin and Galla's equations.}
\
After giving conditions for the realization of a feasible equilibrium and investigating the existence and uniqueness of a stable sub-population (i.e some species vanish), we address the question of estimating the proportion of surviving species as a function of the model paramaters $(\rho,\alpha,  \mu)$.

To our knowledge, this question has not received yet an answer at a mathematical level of rigor and remains open. However theoretical physicists/ecologists provided a solution to this problem supported by simulations. Tools from physics to study population dynamics in the context of Lotka-Volterra equations were first introduced by Opper et al. \cite{diederich_replicators_1989,opper_phase_1992}. 
In 2017, Bunin \cite{bunin2017ecological} precisely answers the question of estimating the proportion of surviving species for the model under investigation (non-centered elliptic model $B$). He uses the dynamical cavity method (a review of which can be found in \cite{barbier_cavity_2017}). The key concept consists of assuming that a unique fixed point exists and introducing a new species with new interactions in the existing system. Provided the coherence of the assumption, an analogy between the properties of the solutions with $n$ and $n+1$ species yields closed-form equations that we present hereafter. 

Notice that recently, similar equations were obtained by Galla \cite{galla_dynamically_2018} using generating functional techniques.

The system of equations presented hereafter is a version of Bunin's equations without the carrying capacity. It is similar to the equations obtained by the replicator equations \cite{diederich_replicators_1989,opper_phase_1992}. Notice that we mention but do not discuss the many implicit assumptions yielding the system of equations.

Let $(\rho,\alpha,  \mu)\in {\mathcal A}$ and $\bs{x}^*$ given by Proposition \ref{prop:unique-equilibrium}. We first introduce the following quantities:
\begin{equation}
    \phi = \frac{\textrm{Card}\{x^*_i>0, i \in [n] \}}{n}\ ,\quad 
    \left \langle \bs{x} \right \rangle = \frac{1}{n}\sum_{j=1}^n x^*_j \ , \quad \left \langle \bs{x}^2 \right \rangle = \frac{1}{n}\sum_{j=1}^n (x_j^*)^2 \ .
\end{equation}
Denote by $Z \sim \mathcal{N}(0,1)$ and set
$$\Delta = (1 +\left \langle \bs{x} \right \rangle  \mu)\frac{\alpha}{\sqrt{\left \langle \bs{x}^2 \right \rangle}}\, .
$$
The following system of 4 equations has 4 unknowns, among which the (supposedly existing) asymptotic limits of $\phi,\left \langle \bs{x} \right \rangle,\left \langle \bs{x}^2 \right \rangle$, denoted (by abuse of notations) by the same notations. The fourth unknown $v$ is a parameter essentially related to the dynamical cavity method. 
This system is supposed to admit a unique solution:
\begin{eqnarray}
 \phi &=& \frac{1}{\sqrt{2\pi}}\int_{-\Delta}^{+\infty}e^{\frac{-z^2}{2}}dz \label{eq:sys_eq_1} \\
 \left \langle \bs{x} \right \rangle &=& \frac{\phi}{1-\frac{\rho v }{\alpha}} \left((1 +\left \langle \bs{x} \right \rangle  \mu)+\frac{\sqrt{\left \langle \bs{x}^2 \right \rangle}}{\alpha}\mathbb{E}(Z|Z>-\Delta) \right) \label{eq:sys_eq_2}\\
\left \langle \bs{x}^2 \right \rangle  &=& \left(\frac{\sqrt{\phi}}{1-\frac{\rho v }{\alpha}} \right)^2 \bigg( \left( 1 +\left \langle \bs{x} \right \rangle  \mu\right)^2 
+2(1 +\left \langle \bs{x} \right \rangle  \mu)\frac{\sqrt{\left \langle \bs{x}^2 \right \rangle}}{\alpha} \mathbb{E}(Z|Z>-\Delta)\nonumber\\
&&\qquad  \ + \ \frac{\left \langle \bs{x}^2 \right \rangle}{\alpha^2} \mathbb{E}(Z^2|Z>-\Delta) \bigg) \label{eq:sys_eq_3}\\
 v &=& \phi \left(\frac{1}{\alpha-\rho v}\right)
\label{eq:sys_eq_4}
\end{eqnarray}
The theoretical solutions of system \eqref{eq:sys_eq_1}-\eqref{eq:sys_eq_4} are compared with the empirical values obtained by Monte-Carlo experiments. As shown in Fig. \ref{fig:theo_vs_emp}, the matching is remarkable.
\begin{figure}[h!]
\centering
\begin{subfigure}[b]{.32\textwidth}
    \includegraphics[width=\textwidth]{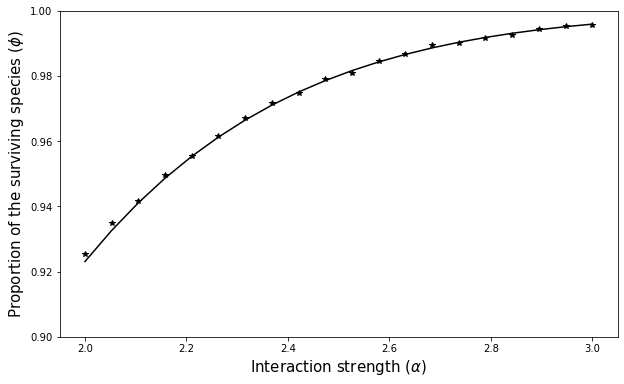}
    \caption{$\phi$ versus $\alpha$,}
    \label{subfig:theo_vs_empi_prop}
  \end{subfigure}%
  \hfill   
  \begin{subfigure}[b]{0.32\textwidth}
    \includegraphics[width=\textwidth]{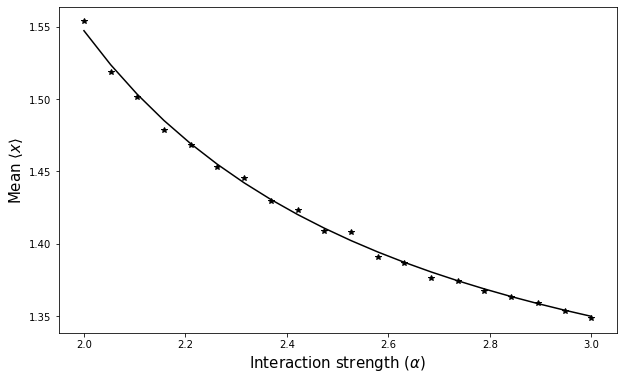}
    \caption{$\left \langle x \right \rangle$ versus $\alpha$,}
    \label{subfig:theo_vs_emp_mean}
  \end{subfigure}%
\hfill   
  \begin{subfigure}[b]{0.32\textwidth}
    \includegraphics[width=\textwidth]{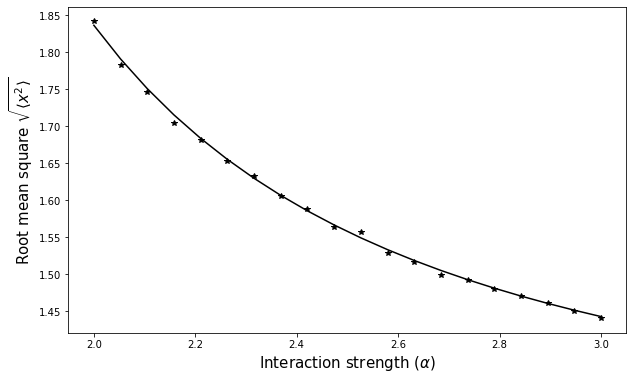}
    \caption{$\left \langle x^2 \right \rangle$ versus $\alpha$.}
    \label{subfig:theo_vs_emp_root}
  \end{subfigure}%
\caption{Theoretical values of $\phi$, $\langle \bs{x}\rangle$ and $\langle \bs{x}^2\rangle$ (solid line) obtained by solving the system \eqref{eq:sys_eq_1}-\eqref{eq:sys_eq_4} given the parameters ($\mu = 0.2,\rho = 0.5$), compared to the empirical values (dots) obtained by Monte-Carlo simulations (size of matrix $n=500$, number of random samples $P=200$). The $x$-axis corresponds to the interaction strength $\alpha$.}
\label{fig:theo_vs_emp}
\end{figure}

The impact of the correlation $\rho$ on the proportion of the surviving species is shown in Figure \ref{fig:cor_impact}.
\begin{remarque}
From a theoretical ecology point of view, notice that a negative correlation (prey-predator) seems to slow down the decline of the surviving species, whereas a positive correlation (mutualism and competition) reverses the trend. 
These types of results are similar to Allesina and Tang \cite{allesina2012stability} where they notice that prey-predator interactions seem to stabilize the system. 
\end{remarque}

\begin{figure}[h!]
  \centering
  \centerline{\includegraphics[width=9.5cm]{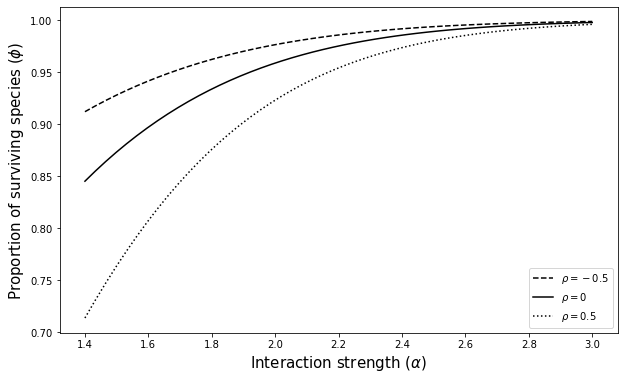}}
   \caption{Effect of the correlation $\rho$ and the interaction strength $\alpha$ on the proportion of surviving species $\phi$. Each curve is plotted by resolving the system \eqref{eq:sys_eq_1}-\eqref{eq:sys_eq_4} in the centered case $\mu = 0$.
   }  
\label{fig:cor_impact}
\end{figure}

\section{Feasibility: Proof of Theorem \ref{th:feasibility}}
\label{sec:proof-feasibility}
We assume that matrix $B_n$ is given by \eqref{eq:elliptic-model} (elliptic model). The case where matrix $B_n$ is 
given by \eqref{eq:cov-model} (covariance profile model) needs extra arguments which are provided in Appendix \ref{app:proof-cov-profile}.
\subsection{Preliminary results}\label{sec:preliminary}
\subsubsection*{Extreme Value Theory (EVT) and the Normal Comparison Lemma}
Let $\left( Z_k \right)_{k\in \mathbb{N}}$ be a sequence of i.i.d. ${\mathcal N}(0,1)$ random variables and denote:
\begin{equation}\label{def:beta}
\begin{cases}M_n=\max_{k\in [n]} Z_k\\ 
\widecheck M_n =\min_{k\in[n]} Z_k
\end{cases}, \quad 
\alpha^*_n= \sqrt{2\log n}\ ,\quad \beta^*_n = \alpha^*_n - \frac 1{2\alpha^*_n}  \log ( 4\pi \log n)\ .
\end{equation}
Let $G(x)=e^{-e^{-x}}$ be the Gumbel cumulative distribution function, then classical EVT results (see for instance \cite[Theorem 1.5.3]{leadbetter2012extremes}) yield that for every $x\in \mathbb{R}$, 
\begin{equation}\label{eq:Gumbel}
    \mathbb{P} \left\{ \alpha^*_n( M_n-\beta^*_n) \le  x \right\} \xrightarrow[n\to\infty]{} G(x)\,,\quad 
\mathbb{P} \left\{ \alpha^*_n(\widecheck M_n +\beta^*_n) \ge -x \right\} \xrightarrow[n\to\infty]{} G(x)\, .
\end{equation}
We consider the following dependent framework: Let $({\mathcal Z}_{k,n})_{k\in [n]}$ be a Gaussian vector whose components are ${\mathcal N}(0,1)$ with covariance
$$
\textrm{cov}\left( {\mathcal Z}_{k,n} ; {\mathcal Z}_{\ell,n}\right) = \frac {\rho}n\,,\quad |\rho|\le 1\,,\quad k\neq \ell\, .
$$
We are interested in the behaviour of 
$
{\mathcal M}_n=\max_{k\in [n]} {\mathcal Z}_{k,n}$ and $
\widecheck {\mathcal M}_n=\min_{k\in [n]} {\mathcal Z}_{k,n}\ ,
$
and shall prove the counterpart of \eqref{eq:Gumbel} with the help of the Normal Comparison Lemma (NCL):
\begin{theo}[Theorem 4.2.1, \cite{leadbetter2012extremes}]
\label{NCL}
Suppose that $(\xi_i,\, i\in [n])$ is a gaussian vector where the $\xi_i$'s are standard normal variables, with covariance matrix $\Lambda^1 = \left(\Lambda_{ij}^1\right)$. Similarly, let $(\eta_i,\, i\in [n])$ be a gaussian vector where the $\eta_i$'s are standard normal, with covariance matrix $\Lambda^0 = \left(\Lambda_{ij}^0\right)$. Denote by $\rho_{ij} = \max\left\{|\Lambda_{ij}^0|,|\Lambda_{ij}^1|\right\}$ and let $(u_i,\, i\in[n])$ be real numbers. Then:

\begin{multline}
    \label{Normal cl}
            \left|P\left\{\xi_j \leq u_j\,,\ j\in [n] \right\} - P\left\{\eta_j \leq u_j\,,\ j\in [n]\right\} \right| \\
            \leq \frac{1}{2\pi}\sum_{1\leq i < j \leq n} \left|\Lambda_{ij}^1 - \Lambda_{ij}^0\right| \left( 1 - \rho_{ij}^2\right)^{-1/2} \exp\left(-\frac{\frac{1}{2}(u_i^2 + u_j^2)}{1+\rho_{ij}}\right)\ .
\end{multline}
\end{theo}

\begin{cor}\label{coro:min-max-dependent}
Recall the definition of $\left( {\mathcal Z}_{k,\ell}\right)_{k\in [n]}$, ${\mathcal M}_n$ and $\widecheck{\mathcal M}_n$ above, then
\begin{equation}\label{eq:Gumbel-DEPENDENT}
    \mathbb{P} \left\{ \alpha^*_n( {\mathcal M}_n-\beta^*_n) \le  x \right\} \xrightarrow[n\to\infty]{} G(x)\,,\quad
\mathbb{P} \left\{ \alpha^*_n(\widecheck {\mathcal M}_n +\beta^*_n) \ge -x \right\} \xrightarrow[n\to\infty]{} G(x)\, .
\end{equation}
\end{cor}

\begin{proof}
We apply the NCL to $(Z_k)_{k\in [n]}$ and $\left({\mathcal Z}_{k,n}\right)_{k\in [n]}$. Let $\rho_{ij} =\frac {|\rho|} n$ and $u_n(x)=\frac x{\alpha_n^*} +\beta_n^*$, then
\begin{eqnarray*}
\lefteqn{\left| \mathbb{P} \{ \alpha^*_n ( M_n - \beta_n^*)\le x\} - \mathbb{P} \{ \alpha^*_n ( {\mathcal M}_n - \beta_n^*)\le x\}\right|\phantom{\bigg|}}\\
&=& \left| \mathbb{P}\{ Z_j \le u_n(x)\, , \ j\in [n]\} - \mathbb{P}\{ {\mathcal Z}_{j,n} \le u_n(x)\, , \ j\in [n]\}\right|\ ,\\
&\le& \frac 1{2\pi} \frac {n(n-1)}2 \frac{|\rho|}n \left( 1-\frac{\rho^2}{n^2}\right)^{-\frac 12} \exp\left( - \frac{u_n^2(x)}{1+\frac {|\rho|} n}\right)\quad \le\quad  K\, n \exp\left( - \frac{u_n^2(x)}{1+\frac 1n}\right)\, .
\end{eqnarray*}
Now eventually $u_n(x)=\alpha_n(1+o(1)) \ge \kappa \alpha_n$ for any $\kappa<1$ and eventually
$$
n\exp \left( - \frac{u_n^2(x)}{1+\frac 1n}\right) \quad \le \quad  n \exp \left( - \frac{2\kappa^2\log(n)}{1+\frac 1n} \right)\quad =\quad 
n^{-\left( \frac{2\kappa^2}{1+\rho/n} - 1\right)}\, .
$$
This last term goes to zero as $n\to \infty$ for a well-chosen $\kappa$ sufficiently close to one. This concludes the proof for ${\mathcal M}_n$. The proof for $\widecheck{\mathcal M}_n$ can be handled similarly with minor modifications.
\end{proof}

\subsubsection*{Random Matrix Theory} 
Let $B_n$ be given by model \eqref{eq:elliptic-model}.
\begin{lemme}\label{lem:spectral-norml-elliptic}
    Let $A_n$ a $n\times n$ matrix with i.i.d. ${\mathcal N}(0,1)$ entries for $i\le j$ and $(A_{ij}, A_{ji})$ a standard bivariate Gaussian vector with covariance $\rho$ for $i<j$, then the following estimate holds true: almost surely,
    $$
    \limsup_{n\to \infty} \left\| \frac{A_n}{\sqrt{n}}\right\| \quad \le\quad  \sqrt{2}\left( \sqrt{1+\rho} +\sqrt{1-\rho}\right)\quad \le \quad 2\sqrt{2}\, .
    $$
\end{lemme}

\begin{proof} The proof relies on two arguments: the classical estimate of the asymptotic spectral norm of a Wigner matrix \cite[Th. 5.1]{book-bai-et-al-2010} and the following decomposition of matrix $A_n/\sqrt{n}$ as linear combination of Hermitian Wigner matrices:
\begin{equation}\label{eq:decomp-elliptic}
\frac{A_n}{\sqrt{n}} = \frac{A_n +A_n^{\tran}}{2\sqrt{n}} -\boldsymbol{i} \frac{\left[\boldsymbol{i}\left(A_n - A_n^{\tran}\right)\right]}{2\sqrt{n}}\,,\qquad (\bs{i}^2=-1)\, .
\end{equation}
Notice that both matrices $W^1_n=\frac{A_n +A_n^{\tran}}{2\sqrt{n}}$ and $W^2_n=\frac{\left[\boldsymbol{i}\left(A_n - A_n^{\tran}\right)\right]}{2\sqrt{n}}$ are Wigner matrices, with off-diagonal variances $(i<j)$:
$$
\mathrm{var}\left( \left[\frac{A_n +A_n^{\tran}}2\right]_{ij}\right) = \frac{1+\rho}2
\qquad \textrm{and}\qquad 
\mathrm{var}\left( \left[ \frac{\boldsymbol{i}\left(A_n - A_n^{\tran}\right)}2\right]_{ij} \right)=\frac {1-\rho}2\, .
$$
Hence,
$$
\limsup_n \left\|\frac{A_n}{\sqrt{n}}\right\| \le \limsup_n \| W^1_n\| +\limsup_n \| W_n^2\| =2 \left( \sqrt{\frac {1+\rho}2}+ \sqrt{\frac {1-\rho}2}\right)
$$
An elementary analysis yields $\sqrt{2} (\sqrt{1+\rho}+ \sqrt{1-\rho})\le 2\sqrt{2}$ for $|\rho|\le 1$. \end{proof}

\subsection{Proof of Theorem \ref{th:feasibility} - the centered case $\mu=0$}
\subsubsection*{Some preparation and strategy of proof}
We first prove Theorem \ref{th:feasibility} in the case where $\mu=0$ and focus on the equation
\begin{equation}\label{eq:equilibrium-centered}
\bs{x}_n = \mathbf{1}_n +\frac{A_n}{\alpha_n \sqrt{n}} \bs{x}_n\, .
\end{equation}
By Lemma \ref{lem:spectral-norml-elliptic}, $\limsup_n \| A_n/\sqrt{n}\|$ is a.s. bounded hence
$$
\left\| \frac{A_n}{\alpha_n \sqrt{n}}\right\| \xrightarrow[n\to\infty]{a.s.} 0\ .
$$
As a consequence, the resolvent $Q_n=\left( I_n - A_n/(\alpha_n \sqrt{n})\right)^{-1}$ is a.s. eventually well-defined
and the solution $\bs{x}_n=(x_k)_{k\in [n]}$ of \eqref{eq:equilibrium-centered} writes $\bs{x}_n = Q_n \mathbf{1}_n$.
Denote by $\bs{e}_k$ the $k$th canonical vector of $\mathbb{R}^n$. The following representation holds true (we shall often drop index $n$ in the following)
\begin{eqnarray}
 x_k&=& \bs{e}_k^\tran \bs{x} = \bs{e}_k Q \mathbf{1} =\sum_{\ell=0}^\infty \bs{e}_k^\tran \left( \frac{A}{\alpha\sqrt{n}} \right)^\ell \bs{1}\ ,\nonumber\\
 &=& 1+\frac 1{\alpha} \bs{e}_k^\tran\left( \frac{A}{\sqrt{n}}\right) \mathbf{1} +\frac 1{\alpha^2} 
 \bs{e}_k^\tran \left( \frac{A}{\sqrt{n}}\right)^2 Q\, \mathbf{1}\, . \label{eq:rep-solution}
\end{eqnarray}
Denote by 
\begin{equation}\label{eq:def-Z-R}
Z_{k,n} = \bs{e}_k^\tran\left( \frac{A}{\sqrt{n}}\right) \mathbf{1} = \frac 1{\sqrt{n}} \sum_i A_{ki}
\quad \textrm{and}\quad 
R_{k,n}(A) = \bs{e}_k^\tran \left( \frac{A}{\sqrt{n}}\right)^2 Q\, \mathbf{1}\, .
\end{equation}
Notice that the $Z_{k,n}$'s are standard ${\mathcal N}(0,1)$ however they are not independent as 
$$
\mathrm{cov} (Z_{k,n}, Z_{\ell, n}) = \frac 1n \mathrm{cov}(A_{k\ell}, A_{\ell k}) =\frac{\rho}n\,,\quad k\neq \ell\,.
$$
Introducing $M_n=\max_{k\in [n]} Z_{k,n}$ and $\widecheck M_n =\min_{k\in [n]} Z_{k,n}$, we proved in Corollary \ref{coro:min-max-dependent} that 
\begin{equation}\label{eq:reminder-EVT}
\mathbb{P} \left\{ \alpha^*_n( M_n-\beta^*_n) \le  x \right\}\ ,\  
\mathbb{P} \left\{ \alpha^*_n(\widecheck M_n +\beta^*_n) \ge -x \right\}\ \xrightarrow[n\to\infty]{}\  G(x)\, .
\end{equation}
In the sequel, we often drop $n$ and simply write $R_k(A)$ instead of $R_{k,n}(A)$. Following the same strategy as in \cite{bizeul2021positive}, we notice that \eqref{eq:rep-solution} yields
$$
\left\{ 
\begin{array}{lcl}
\min_{k\in[n]} x_k& \ge & 1+\frac 1{\alpha} \widecheck M +\frac 1{\alpha^2} \min_{k\in[n]} R_{k}(A)\\
\min_{k\in[n]} x_k& \le & 1+\frac 1{\alpha} \widecheck M +\frac 1{\alpha^2} \max_{k\in[n]} R_{k}(A)
\end{array}\right., 
$$
which we can rewrite
\begin{eqnarray*}
\label{eq:useful}
\min_{k\in[n]} x_k &\geq   1+\frac {\alpha^*_n}{\alpha_n}\left( \frac{\widecheck M +\beta^*_n}{\alpha^*_n} - \frac{\beta^*_n}{\alpha^*_n} +\frac {\min_{k\in[n]} R_k(A)}{\alpha^*_n \alpha_n} \right) \\
 &=\ 1+\frac{\alpha^*_n}{\alpha_n} \left( -1+o_P(1) +\frac {\min_{k\in[n]} R_k(A)}{\alpha^*_n \alpha_n} \right)\, ,
\end{eqnarray*}
where we have use the fact that $\frac{\check M + \beta_n}{\alpha_n^*} = o(1)$, cf. \eqref{eq:reminder-EVT}. Similarly, we have:
\begin{equation}
    \min_{k\in[n]} x_k \ \le\ 1+\frac{\alpha^*_n}{\alpha_n} \left( -1+o_P(1) +\frac {\max_{k\in[n]} R_k(A)}{\alpha^*_n \alpha_n} \right)
\end{equation}
The proof in the centered case follows then from the following lemma:
\begin{lemme}\label{lemme:fonda} Let $R_{k,n}(A)$ be defined as in \eqref{eq:def-Z-R} and recall that $\alpha_n \xrightarrow[n\to +\infty]{} +\infty$, then: 
$$
\frac{\max_{k\in [n]} R_{k,n}(A)}{\alpha_n \sqrt{2\log n}}\  \xrightarrow[n\to\infty]{\mathcal P} \ 0\qquad \textrm{and}\qquad 
\frac{\min_{k\in [n]} R_{k,n}(A)}{\alpha_n \sqrt{2\log n}} \ \xrightarrow[n\to\infty]{\mathcal P} \ 0 \ .
$$
\end{lemme}

The remaining of the section is devoted to the proof of Lemma \ref{lemme:fonda}.

\subsubsection*{Lipschitziannity and Gaussian concentration} 
We first introduce a truncated version of $R_{k,n}(A)$.  Let $\eta\in (0,1)$ and $\varphi:\mathbb{R}^+\to [0,1]$ a smooth function satisfying:
\begin{equation}\label{eq:trunc-function}
\varphi(x) =\begin{cases}
1& \textrm{if}\ x\in [0,2\sqrt{2}+\eta]\\
0& \textrm{if}\ x\ge 4
\end{cases}\ ,
\end{equation}
decreasing from 1 to 0 gradually as $x$ goes from $2\sqrt{2}+\eta$ to $4$. Let
\begin{equation}
\label{eq:tilde}
    \widetilde{R}_{k,n}(A) = \varphi_n R_{k,n}(A)\qquad\textrm{where}\qquad \varphi_n =\varphi\left( \left\| \frac{A_n}{\sqrt{n}}\right\| \right)\, .
\end{equation}
Notice that $\widetilde{R}_k(A)$ differs from $R_k(A)$ if $\varphi_n<1$ which happens with vanishing probability as $\mathbb{P}\left\{\varphi_n<1\right\} = \mathbb{P}\left\{s_n>2\sqrt{2}+\eta\right\} \xrightarrow[n\to \infty]{} 0$ by Lemma \ref{lem:spectral-norml-elliptic}. The following lemma is a first step towards Gaussian concentration.
\begin{lemme}
\label{lemme:lipschitz}
Let $\widetilde{R}_k$ defined by \eqref{eq:tilde} and $M$ an $n\times n$ matrix. Then the function
$$
M \mapsto \widetilde{R}_k(M) = \bs{e}_k^\tran \left( \frac{M}{\sqrt{n}}\right)^2 \left( I - \frac{M}{\alpha \sqrt{n}}\right)^{-1}\mathbf{1}
$$ is $K$-Lipschitz, i.e.
\begin{equation}
    \label{eq:lipschitz}
    \left| \widetilde R_k(M) - \widetilde{R}_k(N)\right| \leq K \left\|M-N\right\|_F\
\end{equation}
where $M,N$ are $n\times n$ matrices, $\|M\|_F=\sqrt{\sum_{ij} |M_{ij}|^2}$ is the Frobenius norm and $K$ a constant independent from $k$ and $n$.
\end{lemme}
The second step is to notice that $\widetilde R_k(A)$ (where $A$ has Gaussian entries but with off-diagonal pairwise correlations) can be in fact expressed as a Lipschitz function of i.i.d. ${\mathcal N}(0,1)$ entries.
\begin{lemme}\label{lemma:lipschitz-gaussian-iid} Consider the linear function $\Gamma:\mathbb{R}^{n\times n} \to \mathbb{R}^{n\times n}$ defined by
$$
 \Gamma_{ii}(X)= X_{ii}\qquad \textrm{and}\qquad
\begin{cases}
 \Gamma_{ij}(X) = \sqrt{\frac{1+\rho}2} X_{ij} + \sqrt{\frac{1-\rho}2} X_{ji}\quad (i<j)\ ,\\
 \Gamma_{ji}(X) = \sqrt{\frac{1+\rho}2} X_{ij} - \sqrt{\frac{1-\rho}2} X_{ji}\quad (i<j)\ .
\end{cases}
$$
Then 
\begin{enumerate}
    \item We have $\| \Gamma(X)\|_F\le K_{\rho} \| X\|_F$ where $K_{\rho}=2\sqrt{1+|\rho|}$ hence $\Gamma$ is $K_{\rho}$-Lipschitz.
    \item If matrix $X_n=(X_{ij})$ has i.i.d. ${\mathcal N}(0,1)$ entries, then $A_n=\Gamma(X_n)$ has i.i.d. ${\mathcal N}(0,1)$ entries on and above the diagonal ($i\le j$) and each vector $(A_{ij}, A_{ji})$ is a standard bivariate Gaussian vector with covariance $\rho$ for $i<j$. 
\end{enumerate}
\end{lemme}
The proof is straightforward and is thus omitted. 

A consequence of this lemma is that $\widetilde R_k(A) = \widetilde R_k(\Gamma(X))$ is $K\times K_{\rho}$-Lipschitz.  Applying Tsirelson-Ibragimov-Sudakov inequality \cite[Theorem 5.5]{boucheron2013concentration} finally yields: 
\begin{prop}
\label{proposition:Sudakov}
Let $K$ the Lipschitz constant of Lemma \ref{lemme:lipschitz} and $K_{\rho}=2\sqrt{1+|\rho|}$. Then $$\mathbb{E}\max_{k\in[n]}\left(\widetilde{R}_k(A) - \mathbb{E}\widetilde{R}_k(A)\right) \leq  2\, K_{\rho}\, K \sqrt{\log n}\, .$$
\end{prop}
Details of the proof are similar to those in \cite{bizeul2021positive} and are thus omitted.

\begin{remarque}\label{rem:estimates}
Notice that $\varphi_n\le 1$ and that $\varphi_n=0$ if $\|A/\sqrt{n}\|\ge 4$. In particular, 
$$
\varphi_n \left\| \frac{A}{\sqrt{n}}\right\| \ \le\  4\qquad \textrm{and}\qquad \varphi_n \left\| Q \right\| \ \le\ 
\frac 1{1- 4\alpha^{-1}}\ \le \ 2
$$
for $n$ large enough. For the latter estimate, write $Q=\left( I - \frac{A}{\alpha\sqrt{n}}\right)^{-1}$,
$Q^{-1}Q = I$ and $Q=I +\frac A{\alpha\sqrt{n}}Q$, then apply the triangular inequality.
\end{remarque}

\begin{prop}
\label{prop:exch}
The following estimate $\mathbb{E}\widetilde{R}_k\left(A_n\right) = \mathcal{O}(1)$ holds true, uniformly for $k\in [n]$.
\end{prop}
\begin{proof}
We shall prove that the variables $\widetilde R_k$ have a common distribution for $k\in [n]$, which in particular implies that 
\begin{equation}\label{eq:exchangeable}
\mathbb{E} \widetilde R_k =\mathbb{E} \widetilde R_i\,, \quad \forall k,i\in [n] \qquad \textrm{and}\qquad 
\mathbb{E} \widetilde R_k =\frac 1n \sum_{i\in [n]} \mathbb{E} \widetilde R_i\, .
\end{equation}
Once this fact is established, the proof is straightforward:
$$
\left| \mathbb{E} \widetilde R_k \right| =\left|
\frac 1n \sum_{i\in [n]} \mathbb{E} \widetilde R_i 
\right| =\left| \frac 1n \mathbb{E} 
\varphi_n \mathbf{1}^\tran \left( \frac A{\sqrt{n}}\right)^2 Q\mathbf{1} \right| 
\le \left\| \frac{\mathbf{1}}{\sqrt{n}} \right\|^2 \mathbb{E} \varphi_n \left\|\frac{A}{\sqrt{n}}\right\|^2\| Q\| =\mathcal{O}(1)\, ,
$$
where the last equality follows from the arguments developed in Remark \ref{rem:estimates}.

Let us now establish \eqref{eq:exchangeable}.

Denote by $\Delta_\sigma$ the matrix associated to the permutation $\sigma: [n] \mapsto [n]$ and defined by 
$$
[\Delta_\sigma]_{ij}=\begin{cases}
 1&\textrm{if}\ i=\sigma(j)\\
 0&\textrm{else}
\end{cases}\, .
$$
Notice in particular that $\Delta_\sigma \boldsymbol{e}_i = \boldsymbol{e}_{\sigma(i)}$, $\Delta_{\sigma}\Delta_\tau = \Delta_{\sigma \tau}$ for $\sigma, \tau$ two permutations and $\Delta_{\sigma^{-1}} = \Delta_\sigma^\tran$. Denote by $(i j)$ the transposition swapping $i$ and $j$, i.e. $(i j)i = j$, $(i j)j = i$ and $(i j)\ell = \ell$ otherwise.
We consider $\widecheck A = \Delta_{(ij)} A \Delta_{(ij)}$, that is $\widecheck A$ is obtained by swapping $A$'s $i$th and $j$th column, then the $i$th and $j$th row. Observe that $A$ and $\widecheck{A}$ have the same distribution and so is the case for $R_k(A)$ and $R_k(\widecheck{A})$.

We have $ \Delta_{(ij)}^2= I_n$, implying that $\widecheck A^k = \Delta_{(ij)} A^k \Delta_{(ij)}$ and then
\begin{eqnarray*}
R_i(\widecheck{A}) &=& \bs{e}_i^\tran \sum\limits_{k\geq 2} \left(\frac{\widecheck A}{\alpha \sqrt{n}}\right)^k\mathbf{1} = \bs{e}_i^\tran \Delta_{(ij)} \sum\limits_{k\geq 2} \left(\frac{A}{\alpha \sqrt{n}}\right)^k \Delta_{(ij)}\mathbf{1} = \bs{e}_j^\tran \sum\limits_{k\geq 2} \left(\frac{A}{\alpha \sqrt{n}}\right)^k\mathbf{1} \\
&=& R_j(A)\, .
\end{eqnarray*}
This proves that $R_i(A), R_i(\widecheck A), R_j(A)$ have the same law, hence the same expectation. Eq.\eqref{eq:exchangeable} is established, which concludes the proof.
\end{proof}

We are now in position to prove Lemma \ref{lemme:fonda}.

\begin{proof}[Proof of lemma \ref{lemme:fonda}]
Recall that $\mathbb{E}\widetilde R_k (A) = \mathbb{E}\widetilde R_1$.  Since $\max_{k\in[n]}\widetilde R_k(A) - \widetilde R_1(A) \geq 0$, Markov inequality yields:

\begin{eqnarray*}
\mathbb{P} \left\{ \frac{\max_{k\in [n]} \widetilde R_k(A) - \widetilde R_1(A)}{\alpha \sqrt{2\log n}} \ge  \varepsilon \right\} &
\le&\frac{\mathbb{E} \left(\max_{k\in [n]} \widetilde R_k(A) - \widetilde R_1(A)\right)}{\varepsilon \alpha \sqrt{2\log n}}\,,\\
&=& \frac{\mathbb{E}\left(\max_{k\in [n]} \widetilde R_k\left(A\right) - \mathbb{E}\widetilde R_k(A)\right)}{\varepsilon \alpha \sqrt{2\log n}}\,,\\ & =& \frac{\mathbb{E}\left(\max_{k\in [n]} \left(\widetilde R_k\left(A\right) - \mathbb{E}\widetilde R_k(A)\right)\right)}{\varepsilon \alpha \sqrt{2\log n}}\,,\\
& \le & \frac{\sqrt 2 K\times K_{\rho}}{\varepsilon \alpha}\ ,
\end{eqnarray*}
where the last inequality follows from Proposition \ref{proposition:Sudakov}.\\
This implies that
$$
\frac{\max_{k\in[n]} \widetilde{R}_k\left(A\right) -\widetilde{R}_1\left(A\right)} {\alpha\sqrt{2 \log n}} \ \xrightarrow[n\to \infty]{\mathcal{P}}\  0\,.
$$
It remains to prove that 
$$
    \frac{\widetilde{R}_1\left(A\right)}{\alpha \sqrt{2\log n}}\ \xrightarrow[n \to \infty]{\mathcal{P}}\ 0\qquad \textrm{and}\qquad \frac{\max_{k\in[n]} R_k\left(A\right)} {\alpha\sqrt{2 \log n}}\ \xrightarrow[n\to \infty]{\mathcal{P}}\ 0\, .
$$
The arguments are similar to those in \cite[Section 2.3]{bizeul2021positive}. 
Proof of the second assertion of Lemma \ref{lemme:fonda} can be done similarly. This concludes the proof.
\end{proof}

\subsection{Proof of Theorem \ref{th:feasibility} - the non centered case.}

Recall that $\alpha\to \infty$ as $n\to \infty$. Denote by $\bs{u}_n = \frac{1}{\sqrt n} \mathbf{1}_n$ and notice that the spectrum of $I_n -\mu \bs{u}_n \bs{u}_n^\tran$ is $\{1-\mu, 1\}$, the eigenvalue $1$ with multiplicity $n-1$. Notice in particular that if $\mu\neq 1$, then $I-\mu \bs{u} \bs{u}^\tran$ is invertible. So is (eventually) $I - \frac{A}{\alpha\sqrt{n}} - \mu \bs{u}\bs{u}^\tran$ as $\| A/(\alpha\sqrt{n})\|\to 0$ a.s. We shall also rely on the fact that $\| Q - I\| \xrightarrow[n\to\infty]{} 0$ a.s. As a consequence, 
$$
\bs{u}^\tran Q \bs{u} \xrightarrow[n\to\infty]{a.s.} 1\, .
$$
Denote by $\bs{\tilde x}$ and $\bs{x}$ the vectors solutions of the equations:
$$
\bs{\tilde x} = \mathbf{1} +B\bs{\tilde x}= \mathbf{1}+\left( \frac{A}{\alpha\sqrt{n}} +\mu \bs{u} \bs{u}^\tran\right)\bs{\tilde x}
\qquad \textrm{and}\qquad 
\bs{x} = \mathbf{1} +\frac{A}{\alpha\sqrt{n}} \bs{x}\, .
$$
The following representations hold:
$$
\bs{\tilde x} =\left( I - B\right)^{-1} \mathbf{1} \qquad \textrm{and}\qquad \bs{x}=\left(I - \frac{A}{\alpha\sqrt{n}} \right)^{-1} \mathbf{1}\, .
$$
Recall that $Q=\left(I-A/(\alpha\sqrt{n})\right)^{-1}$. By rank one perturbation identity (Woodbury), we have:
$$
(I-B)^{-1}= Q  +\frac{Q\bs{u} \bs{u}^\tran Q}{1-\mu \bs{u}^\tran Q \bs{u}}
$$
and
$$
\bs{\tilde x} \quad =\quad  \frac{Q\mathbf{1} (1-\mu \bs{u}^\tran Q\bs{u}) + \mu Q \bs{u} \bs{u}^\tran Q \mathbf{1}}{1-\mu \bs{u}\tran Q \bs{u}}\quad =\quad \frac{\bs{x}}{1-\mu \bs{u}^\tran Q \bs{u}}\, .
$$
If $\mu<1$ and $\alpha\ge (1+\varepsilon)\alpha^*$ then eventually, $\bs{\tilde x}$ has positive components. This is no longer the case if $\mu>1$ or $\alpha\le (1-\varepsilon)\alpha^*$. This concludes the proof of Theorem \ref{th:feasibility}.

\section{Stability: Proof of Proposition \ref{prop:unique-equilibrium}}
\label{sec:proof-stability}

\begin{proof}
We have
$$
I- B +I -B^\tran = 2I - (B +B^\tran)
= 2I - \left(\frac{A+A^\tran}{\alpha \sqrt{n}}+\frac{2\mu}{n} \bs{1}\bs{1}^\tran\right)\, .
$$
We will rely on the following condition:
\begin{equation}
\label{cond_eig}
2I - (B +B^\tran) \text{ is positive definite } \quad \Leftrightarrow \quad \lambda_{\max}(B+B^\tran) < 2 \, .
\end{equation}
Notice that $(A+A^\tran)/\alpha$ is a symmetric matrix with independent ${\mathcal N}(0,2(1+\rho)/\alpha^2)$ entries above the diagonal (the diagonal entries have a different distribution from the off-diagonal entries, with no asymptotic effect). In this case, it is well known that the largest eigenvalue of the normalized matrix (or equivalently its spectral norm since the matrix is symmetric) almost surely converges to the right edge of the support of the semi-circle law (see \cite[Theorem 5.2]{book-bai-silverstein}):
\begin{equation}
\label{centered_cond}
\lambda_{\max}\left( \frac{A+A^\tran}{\alpha\sqrt{ n}}
\right)
\quad \xrightarrow[n\to\infty]{a.s.} \quad \frac{2 \sqrt{2(1+\rho)}}{\alpha}\, .
\end{equation}
Suppose that $(\rho, \alpha, \mu) \in {\mathcal A}$. Notice that in this case, 
$$
\frac{\sqrt{1+\rho}}{\alpha\sqrt{2}} \quad <\quad \frac 12 \quad <\quad  \frac 12 +\frac 12 \sqrt{1 - \frac{2(1+\rho)}{\alpha^2}}\, .
$$
We consider three subcases 
\begin{enumerate}
    \item[(i)] $\mu=0$,
    \item[(ii)] $\mu \le \frac{\sqrt{1+\rho}}{\alpha \sqrt{2}}$, 
    \item[(iii)] $
\mu \in \left( \frac {\sqrt{1+\rho}}{\alpha\sqrt{2}}\ ,\ \frac{1}{2}+\frac 12 \sqrt{1-\frac{2(1+\rho)}{\alpha^2}}\right)\ .
$
\end{enumerate}

In the centered case (i), condition (\ref{cond_eig}) asymptotically occurs whenever $\alpha > \sqrt{2(1+\rho)}$.

Before studying subcases (ii) and (iii), we recall a result on small rank perturbations of large random matrices. 

Notice that the rank-one perturbation matrix $P = \frac{2\mu}{n} \bs{1}\bs{1}^\tran$ admits a unique non zero eigenvalue $2\mu$. Denote by $\check{A}=\frac{A+A^\tran}{\alpha\sqrt{n}}$. We are concerned with the top eigenvalue of the symmetric matrix $\check{A}+P$. Based on a result by Capitaine et al. \cite[Theorem 2.1]{capitaine_largest_2009}, we have:
\begin{equation*}
    \lambda_{\max}(\check{A}+P) \xrightarrow[n \rightarrow \infty]{a.s}\left\{ 
 \begin{array}{ll}
  2\mu+\frac{1+\rho}{\alpha^2  \mu} &\text{if } \mu> \frac{\sqrt{1+\rho}}{\sqrt{2} \alpha}\, ,  \\ 
  \frac{2 \sqrt{2(1+\rho)}}{\alpha} &\text{else.}
\end{array}\right.
\end{equation*}
Consider now subcase (ii), then $\lambda_{\max}( \check{A}+P) \xrightarrow[n\to\infty]{a.s.} \frac{2\sqrt{2(1+\rho)}}{\alpha}$, which is strictly lower than 2 since $(\rho,\alpha, \mu)\in {\mathcal A}$. Hence $\lambda_{\max}(\check{A} +P)$ is eventually strictly lower than 2 in this case.

We finally consider subcase (iii). In this case, 
$$
\lambda_{\max}( \check{A}+P) \xrightarrow[n\to\infty]{a.s.} 2\mu+\frac{1+\rho}{\alpha^2  \mu}\, .
$$
We shall prove that $2\mu +\frac {1+\rho}{\alpha^2 \mu}<2$ or equivalently 
\begin{equation}\label{eq:condition-poly}
    2\alpha^2 \mu^2 - 2\alpha^2 \mu +1+\rho<0\, .
\end{equation}
An elementary study of the polynomial $\Omega(X)= 2\alpha^2 X^2 - 2\alpha^2 X +1+\rho$ yields that $\Omega$'s discriminant is positive if $\alpha>\sqrt{2(1+\rho)}$ and $\Omega$'s roots are given by 
$$\Omega(\mu^{\pm})=0\quad \Leftrightarrow\quad  \mu^{\pm}= \frac 12 \pm \frac 12 \sqrt{1- \frac{2(1+\rho)}{\alpha^2}}\ .
$$ 
Also remark that $\Omega\left( \frac {\sqrt{1+\rho}}{\alpha\sqrt{2}}\right) <0$, so that $\frac {\sqrt{1+\rho}}{\alpha\sqrt{2}}\in (\mu^-, \mu^+)$. In particular condition \eqref{eq:condition-poly} is fulfilled for $\mu \in \left(\frac {\sqrt{1+\rho}}{\alpha\sqrt{2}},\mu^+\right) $, which is precisely subcase (iii).
 Hence a.s. $\limsup_{n\to\infty}\lambda_{\max}(\check{A}+P) < 2$. We can then rely on Theorem \ref{th:takeuchi} to conclude.

\end{proof}


\bibliographystyle{abbrv}
\bibliography{mathematics}

\begin{thebibliography}{10}

\bibitem{akjouj2021feasibility}
I.~Akjouj and J.~Najim.
\newblock Feasibility of sparse large lotka-volterra ecosystems.
\newblock {\em arXiv preprint arXiv:2111.11247}, 2021.

\bibitem{allesina2012stability}
S.~Allesina and S.~Tang.
\newblock Stability criteria for complex ecosystems.
\newblock {\em Nature}, 483(7388):205, 2012.

\bibitem{Allesina:2015ux}
S.~Allesina and S.~Tang.
\newblock The stability--complexity relationship at age 40: a random matrix
  perspective.
\newblock {\em Population Ecology}, 57(1):63--75, 2015.

\bibitem{book-bai-et-al-2010}
Z.~Bai and J.~W. Silverstein.
\newblock {\em Spectral analysis of large dimensional random matrices}.
\newblock Springer Series in Statistics. Springer, New York, second edition,
  2010.

\bibitem{book-bai-silverstein}
Z.~D. Bai and J.~W. Silverstein.
\newblock {\em Spectral analysis of large dimensional random matrices}.
\newblock Springer Series in Statistics. Springer, New York, second edition,
  2010.

\bibitem{barbier_cavity_2017}
M.~Barbier and J.-F. Arnoldi.
\newblock The cavity method for community ecology.
\newblock preprint, Ecology, June 2017.

\bibitem{bizeul2021positive}
P.~Bizeul and J.~Najim.
\newblock Positive solutions for large random linear systems.
\newblock {\em Proceedings of the American Mathematical Society},
  149(6):2333--2348, 2021.

\bibitem{boucheron2013concentration}
S.~Boucheron, G.~Lugosi, and P.~Massart.
\newblock {\em Concentration inequalities: A nonasymptotic theory of
  independence}.
\newblock Oxford university press, 2013.

\bibitem{bunin2017ecological}
G.~Bunin.
\newblock Ecological communities with lotka-volterra dynamics.
\newblock {\em Physical Review E}, 95(4):042414, 2017.

\bibitem{capitaine_largest_2009}
M.~Capitaine, C.~Donati-Martin, and D.~Féral.
\newblock The largest eigenvalues of finite rank deformation of large {Wigner}
  matrices: {Convergence} and nonuniversality of the fluctuations.
\newblock {\em The Annals of Probability}, 37(1), Jan. 2009.

\bibitem{code}
M.~Clenet.
\newblock Feasibility in a large lotka-volterra system with pairwise correlated
  interactions.
\newblock
  \textit{https://github.com/maxime-clenet/Feasibility-in-a-large-Lotka-Volterra-system-with-pairwise-correlated-interactions},
  2022.

\bibitem{clenet2022preprint}
M.~Clenet, F.~Massol, and J.~Najim.
\newblock Equilibrium and surviving species in a large lotka-volterra system of
  differential equations.
\newblock {\em Submitted. arxiv:2205.00735}, 2022.

\bibitem{diederich_replicators_1989}
S.~Diederich and M.~Opper.
\newblock Replicators with random interactions: {A} solvable model.
\newblock {\em Physical Review A}, 39(8):4333--4336, Apr. 1989.

\bibitem{dougoud2018feasibility}
M.~Dougoud, L.~Vinckenbosch, R.~P. Rohr, L.-F. Bersier, and C.~Mazza.
\newblock The feasibility of equilibria in large ecosystems: A primary but
  neglected concept in the complexity-stability debate.
\newblock {\em PLoS computational biology}, 14(2):e1005988, 2018.

\bibitem{galla_dynamically_2018}
T.~Galla.
\newblock Dynamically evolved community size and stability of random
  {Lotka}-{Volterra} ecosystems.
\newblock {\em EPL (Europhysics Letters)}, 123(4):48004, Sept. 2018.
\newblock arXiv: 1808.06660.

\bibitem{gardner1970connectance}
M.~R. Gardner and W.~R. Ashby.
\newblock Connectance of large dynamic (cybernetic) systems: critical values
  for stability.
\newblock {\em Nature}, 228(5273):784, 1970.

\bibitem{gibbs2018effect}
T.~Gibbs, J.~Grilli, T.~Rogers, and S.~Allesina.
\newblock Effect of population abundances on the stability of large random
  ecosystems.
\newblock {\em Physical Review E}, 98(2):022410, 2018.

\bibitem{girko1986elliptic}
V.~Girko.
\newblock Elliptic law.
\newblock {\em Theory of Probability \& Its Applications}, 30(4):677--690,
  1986.

\bibitem{girko1995elliptic}
V.~Girko.
\newblock The elliptic law: ten years later i.
\newblock 1995.

\bibitem{gopalsamy1984global}
K.~Gopalsamy.
\newblock Global asymptotic stability in volterra's population systems.
\newblock {\em JOURNAL of Mathematical Biology}, 19(2):157--168, 1984.

\bibitem{hering1990oscillations}
R.~Hering.
\newblock Oscillations in lotka-volterra systems of chemical reactions.
\newblock {\em Journal of mathematical chemistry}, 5(2):197--202, 1990.

\bibitem{hofbauer1998evolutionary}
J.~Hofbauer and K.~Sigmund.
\newblock {\em Evolutionary games and population dynamics}.
\newblock Cambridge university press, 1998.

\bibitem{janson1997gaussian}
S.~Janson.
\newblock {\em Gaussian hilbert spaces}.
\newblock Number 129. Cambridge university press, 1997.

\bibitem{kiss2008qualitative}
K.~Kiss and S.~Kov{\'a}cs.
\newblock Qualitative behavior of n-dimensional ratio-dependent predator--prey
  systems.
\newblock {\em Applied Mathematics and Computation}, 199(2):535--546, 2008.

\bibitem{latala2005some}
R.~Lata{\l}a.
\newblock Some estimates of norms of random matrices.
\newblock {\em Proceedings of the American Mathematical Society},
  133(5):1273--1282, 2005.

\bibitem{leadbetter2012extremes}
M.~R. Leadbetter, G.~Lindgren, and H.~Rootz{\'e}n.
\newblock {\em Extremes and related properties of random sequences and
  processes}.
\newblock Springer Science \& Business Media, 2012.

\bibitem{may1972will}
R.~May.
\newblock Will a large complex system be stable?
\newblock {\em Nature}, 238(5364):413, 1972.

\bibitem{naumov2012elliptic}
A.~Naumov.
\newblock Elliptic law for real random matrices.
\newblock {\em arXiv preprint arXiv:1201.1639}, 2012.

\bibitem{nguyen2015elliptic}
H.~H. Nguyen and S.~O’Rourke.
\newblock The elliptic law.
\newblock {\em International Mathematics Research Notices},
  2015(17):7620--7689, 2015.

\bibitem{opper_phase_1992}
M.~Opper and S.~Diederich.
\newblock Phase transition and 1/f noise in a game dynamical model.
\newblock {\em Physical Review Letters}, 69(10):1616--1619, Sept. 1992.
\newblock Publisher: American Physical Society.

\bibitem{orourke2014low}
S.~O'Rourke and D.~Renfrew.
\newblock {Low rank perturbations of large elliptic random matrices}.
\newblock {\em Electronic Journal of Probability}, 19(none):1 -- 65, 2014.

\bibitem{stone2018feasibility}
L.~Stone.
\newblock The feasibility and stability of large complex biological networks: a
  random matrix approach.
\newblock {\em Scientific reports}, 8(1):8246, 2018.

\bibitem{takeuchi1996global}
Y.~Takeuchi.
\newblock {\em Global dynamical properties of Lotka-Volterra systems}.
\newblock World Scientific, 1996.

\end{thebibliography}

\hspace{1cm}

\noindent {\sc Maxime Clenet},\\
Laboratoire d'Informatique Gaspard Monge, UMR 8049\\
CNRS \& Universit\'e Gustave Eiffel\\
5, Boulevard Descartes,\\
Champs sur Marne,
77454 Marne-la-Vall\'ee Cedex 2, France\\
e-mail: {\tt maxime.clenet@univ-eiffel.fr}\\

\noindent {\sc Hafedh El Ferchichi},\\
Ecole Normale Supérieure Paris Saclay\\
e-mail: {\tt hafedh.el\_ferchichi@ens-paris-saclay.fr}\\

\noindent {\sc Jamal Najim},\\
Laboratoire d'Informatique Gaspard Monge, UMR 8049\\
CNRS \& Universit\'e Gustave Eiffel\\
5, Boulevard Descartes,\\
Champs sur Marne,
77454 Marne-la-Vall\'ee Cedex 2, France\\
e-mail: {\tt najim@univ-mlv.fr}\\

\appendix

\section{Proof of Theorem \ref{th:feasibility}: adaptations to the case of a covariance profile}
\label{app:proof-cov-profile}
In this section, we provide the arguments to prove Theorem \ref{th:feasibility} in the case where matrix $B$ follows the model \eqref{eq:cov-model}, i.e.
$$
B=\frac{\tilde A_n}{\alpha_n \sqrt{n}} + \frac{\mu}n \mathbf{1}_n \mathbf{1}_n^\tran\, , 
$$
where $\tilde A_n$'s entries are i.i.d. ${\mathcal N}(0,1)$ on and above the diagonal ($i\le j$), and $(\tilde A_{ij}, \tilde A_{ji})$ is a standard bivariate Gaussian vector $(i<j$) with covariance $\mathrm{cov}(\tilde A_{ij}, \tilde A_{ji})=\rho_{ij}$, and independent from the remaining random variables.

There are essentially 3 issues to resolve, to fully adapt the proof developed in Section \ref{sec:proof-feasibility} to the covariance profile case:
\begin{enumerate}
    \item \label{item:spectral-norm} The decomposition \eqref{eq:decomp-elliptic} yields $\frac{\tilde A_n}{\sqrt{n}}= W_n^1+\bs{i} W_n^2$, where $W_n^1, W_n^2$ are Hermitian matrices with
    $$
     \mathrm{var}\left( [W_n^1]_{ij}\right) =\frac{1+\rho_{ij}}{2n}\qquad \textrm{and}\qquad 
     \mathrm{var}\left( [W_n^2]_{ij}\right) =\frac{1-\rho_{ij}}{2n}\, .
    $$
    Since $W_n^1, W_n^2$ are no longer Wigner matrices, but rather matrices with a variance profile, an extra argument is needed to obtain an almost-sure upper bound for
    $\limsup_{n} \|W_n^1\| + \limsup_{n} \|W_n^2\|$.
    \item \label{item:lipschitz} The Lipschitz property for $\widetilde{R}_{k,n}(\tilde A_n)$. Essentially, we need the counterpart of Lemma \ref{lemma:lipschitz-gaussian-iid} to the context of a covariance profile. 
    \item \label{item:mean} The control of the term $\mathbb{E}\, \widetilde{R}_{k,n}(\tilde A)$.
\end{enumerate}

\subsection{Proof of issue \ref{item:spectral-norm}: Control of the spectral norm of a Hermitian matrix with a variance profile}

Applying Lata\l a's theorem \cite{latala2005some}, we easily show that 
$$
\mathbb{E} \| W_n^1\| +  \mathbb{E}\|W_n^2\| \le C\, 
$$
where $C$ is a constant independent from $n$.

Now write 
$$W_n^1= \frac{\Upsilon_n \circ X_n}{\sqrt{n}} \qquad \textrm{where}\qquad \Upsilon_n =\left( \Upsilon_{ij} \right)\ ,\quad \Upsilon_{ij}= \sqrt{\frac{1+\rho_{ij}}{2}} ,
$$ matrix $X_n=(X_{ij})$ is a Wigner matrix with i.i.d. ${\mathcal N}(0,1)$ entries on and above the diagonal, and $\circ$ stands for the Hadamard product, i.e. $\Upsilon_n \circ X_n = \left( \Upsilon_{ij} X_{ij}
\right)$. Notice that $\sqrt{n} W_n^1$ is 1-Lipschitz with respect to the Frobenius norm 
$$
\| X_n\|_{\textrm{Frob}} = \sqrt{\sum_{ij} |X_{ij}|^2}\ .
$$
Hence by Gaussian concentration, we have
$$
\mathbb{P} \left\{ 
\left| \sqrt{n} \| W_n^1\| - \sqrt{n} \mathbb{E} \|W_n^1\| \right| >\delta 
\right\}\le 2e^{-\frac{\delta^2}{2}}\, .
$$
Taking $\delta=\varepsilon\sqrt{n}$, we obtain
$$
\mathbb{P} \left\{ 
\left|  \| W_n^1\| -  \mathbb{E} \|W_n^1\| \right| >\varepsilon
\right\}\le 2e^{-\frac{n\varepsilon^2}{2}}\, .
$$
The same holds for $W_n^2$, hence the upper control:
$$
\limsup_n \left( \| W_n^1\| +\|W_n^2\| \right) \quad \le\quad  \limsup_n \left( 
\mathbb{E} \| W_n^1\| +\mathbb{E}\|W_n^2\|\right) \quad \le\quad  C
$$
almost surely. It remains to replace the truncation function $\varphi$ in \eqref{eq:trunc-function} by the smooth function
$$
\psi(x) = \begin{cases}
1 &\textrm{if}\ x\le C+\eta,\\
0&\textrm{else.}
\end{cases}
$$
to proceed.
\subsection{Proof of issue \ref{item:lipschitz}: $\widetilde{R}_k(\tilde A)$ is a Lipschitz function of Gaussian i.i.d. random variables}
It suffices to replace function $\Gamma$ in Lemma \ref{lemma:lipschitz-gaussian-iid} by
$$
\widetilde{\Gamma}:\mathbb{R}^{n\times n} \rightarrow \mathbb{R}^{n\times n}$$
where
$$
\widetilde{\Gamma}_{ii}(X)= X_{ii}\qquad \textrm{and}\qquad
\begin{cases}
 \widetilde{\Gamma}_{ij}(X) = \sqrt{\frac{1+\rho_{ij}}2} X_{ij} + \sqrt{\frac{1-\rho_{ij}}2} X_{ji}\quad (i<j)\ ,\\
 \widetilde{\Gamma}_{ji}(X) = \sqrt{\frac{1+\rho_{ij}}2} X_{ij} - \sqrt{\frac{1-\rho_{ij}}2} X_{ji}\quad (i<j)\ .
\end{cases}
$$
and to modify accordingly the Lipschitz constant by $\widetilde{K}=2\sqrt{2}\ge 2\sqrt{1+\max_{ij} |\rho_{ij}|}$.

\subsection{Proof of issue \ref{item:mean}: Magnitude of $\mathbb{E}\, \widetilde{R}_{k,n}(\tilde A_n)$}

To address this issue, we provide a quick argument which relies on Isserlis' theorem also called Wick's formula (see \cite[Th. 1.28]{janson1997gaussian}), highly dependent on the Gaussiannity of the entries. 

\begin{theo}[Isserlis Theorem]
if $(X_1,\cdots, X_n)$ is a centered normal vector, then
\begin{equation}
    \mathbb{E}(X_1X_2\cdots X_n) = \sum_{\Pi} \prod_{\{i,j\}\in \Pi} \mathbb{E}(X_iX_j)
\end{equation}
where the sum is over all the partitions $\Pi$ of $[n]$ into pairs $\{i,j\}$, and the product over all the pairs contained in $\Pi$.
\end{theo}
Recall that:
$$
    \mathbb{E}\widetilde{R}_k(\tilde A_n) = \sum_{\ell\geq 2} \mathbb{E} \left[ \frac {\bs{e}_k^\tran }{\alpha^{\ell-2} } \left(\frac{\tilde A}{\sqrt{n}}\right)^\ell \mathbf{1}\right] 
    =  \sum_{\ell\geq 2} \frac {\mathbb{E} \bs{e}_k^\tran \tilde A^\ell \mathbf{1}}{\alpha^{\ell-2} n^{\frac \ell 2}}  = :  \sum_{\ell\geq 2} \frac {C_\ell}{\alpha^{\ell-2} n^{\frac \ell 2}} \, .
$$
Consider a matrix $\overline{A}_n$ where the pairwise covariance $\mathrm{cov}(\overline{A}_{ij},\overline{A}_{ji})=1$. Denote by $\overline{C}_\ell= \mathbb{E} \bs{e}_k^\tran\, \overline{A}^\ell\, \mathbf{1}$. We will show that each quantity $|C_{\ell}|$ is bounded by $\overline{C}_{\ell}$. 
Notice that:
\begin{equation}
    C_\ell = \sum_{i_1,...,i_{\ell+1}} \mathbb{E} (\tilde A_{i_1i_2} \tilde A_{i_2i_3}...\tilde A_{i_{\ell},i_{\ell+1}})\,.
\end{equation}
by Isserlis' theorem, we have:
\begin{eqnarray*}
    \left|\mathbb{E} (\tilde A_{i_1,i_2} \tilde A_{i_2,i_3}...\tilde A_{i_{\ell},i_{\ell+1}})\right| &\leq& \left|\sum_{\Pi} \prod_{\{j,k\}\in \Pi} \mathbb{E}(\overline{A}_{i_ji_{j+1}} \overline{A}_{i_ki_{k+1}}) \right| \\
    &\leq& \sum_{\Pi} \prod_{\{j,k\}\in \Pi} \mathbb{E}(\overline{A}_{i_ji_{j+1}} \overline{A}_{i_ki_{k+1}})\quad=\quad \mathbb{E} \left(\overline{A}_{i_1i_2} \cdots \overline{A}_{i_{\ell}i_{\ell+1}}\right)\, .
\end{eqnarray*}
From this, we deduce that $|C_{\ell}| \le \overline{C}_{\ell}$, hence $|\mathbb{E} \widetilde{R}_k(\tilde A)| \le \mathbb{E} \widetilde{R}_k(\overline{A})$. This gives the desired bound since $\mathbb{E}\widetilde{R}_k(\overline{A})= {\mathcal O}(1)$.

\end{document}